\documentclass[a4paper,11pt]{amsart}
\usepackage[T1]{fontenc}
\usepackage[all]{xy}
\usepackage{amssymb}
\usepackage[english,french]{babel}

\newtheorem
{theorem}{Theorem}[section]
\newtheorem{proposition}[theorem]{Proposition}

\newtheorem{remark}[theorem]{Remark}
\newtheorem{definition}[theorem]{Definition}
\newtheorem*{theorem*}{Theorem}
\newtheorem*{conjecture*}{Conjecture}

\numberwithin{equation}{section}

\newcommand{\C}{{\mathbb C} }
\newcommand{\R}{{\mathbb R} }
\newcommand{\cA}{{\mathcal A} }

\newcommand{\cE}{{\mathcal E} }
\newcommand{\cF}{{\mathcal F} }
\newcommand{\cG}{{\mathcal G} }
\newcommand{\cI}{{\mathcal I} }

\newcommand{\cK}{{\mathcal K} }
\newcommand{\cL}{{\mathcal L} }
\newcommand{\cM}{{\mathcal M} }

\newcommand{\cO}{{\mathcal O} }

\newcommand{\cZ}{{\mathcal Z} }

\newcommand{\wt}{\widetilde}
\newcommand{\wh}{\widehat}

\def\ol#1{{\overline{#1}}}
\def\ul#1{{\underline{#1}}}
\def\ii{\sqrt{-1}}
\def\pt{\partial}

\def\cinf{C^\infty}
\def\we{\wedge}

\def\ka{K\"ah\-ler}
\def\he{Her\-mite-Ein\-stein}
\def\wp{Weil-Pe\-ters\-son}
\def\ks{Kodaira-Spencer}
\def\psh{plurisubharmonic}

\makeatletter
\newenvironment{abstracts}{%
  \ifx\maketitle\relax
    \ClassWarning{\@classname}{Abstract should precede
      \protect\maketitle\space in AMS document classes; reported}%
  \fi
  \global\setbox\abstractbox=\vtop \bgroup
    \normalfont\Small
    \list{}{\labelwidth\z@
      \leftmargin3pc \rightmargin\leftmargin
      \listparindent\normalparindent \itemindent\z@
      \parsep\z@ \@plus\p@
      
      \itemsep\medskipamount
    }%
}{%
  \endlist\egroup
  \ifx\@setabstract\relax \@setabstracta \fi
}

\newcommand{\abstractin}[1]{%
  \otherlanguage{#1}%
  \item[\hskip\labelsep\scshape\abstractname.]%
}
\makeatother

\begin{document}

\title[Weil-Petersson current for moduli of vector bundles]{The Weil-Petersson current\\ for moduli of vector bundles\\ and applications to orbifolds
}
\author[I.\ Biswas]{Indranil Biswas}

%\selectlanguage{english}
\address{School of Mathematics, Tata Institute of Fundamental Research, Homi Bhabha Road, Mumbai 400005, India}
\email{indranil@math.tifr.res.in}
\author[G.\ Schumacher]{Georg Schumacher}

\address{Fachbereich Mathematik und Informatik,
Philipps-Universität Marburg, \break Lahnberge, Hans-Meerwein-Stra\ss e, D-35032
Marburg, Germany} \email{schumac@mathematik.uni-marburg.de}
%\maketitle
\begin{abstracts}
  \abstractin{french}
  Nous étudions les fibrés vectoriels holomorphes stables sur une varieté kählérienne compacte ou plus généralement sur une orbifold possédant une structure kählérienne. Dans ce contexte, nous utilisons l'existence d'une connection Hermite-Einstein et construisons une forme de Weil-Petersson généralisée sur l'espace des modules des fibrés holomorphes stables à fibré déterminant fixé. Nous montrons que la forme de Weil-Petersson s'étend en un courant (semi-)positif fermé pour des dégénerescences de familles qui sont des restrictions de faisceaux cohérents. Ce courant sera appelé un courant de Weil-Petersson. Dans le cas d'une orbifold de type Hodge, un fibré en droite déterminant existe sur l'espace des modules. Ce fibré en droites est muni d'une métrique de Quillen dont la courbure coincide avec la forme de Weil-Petersson généralisée. En application, nous montrons que le fibré en droites déterminant s'étend à une compactification de l'espace des modules.
  \abstractin{english}
  We investigate stable holomorphic vector bundles on a compact complex \ka\ manifold and more generally on an orbifold that is  equipped with a \ka\ structure. We use the existence of Hermite-Einstein connections in this set-up and construct a generalized \wp\ form on the moduli space of stable vector bundles with fixed determinant bundle. We show that the \wp\ form extends as a (semi-)positive closed current for degenerating families that are restrictions of coherent sheaves. Such an extension will be called a \wp\ current. When the orbifold is of Hodge type, there exists a certain determinant line bundle on the moduli space; this line bundle carries a Quillen metric, whose curvature coincides with the generalized \wp\ form. As an application we show that the determinant line bundle extends to a suitable compactification of the moduli space.
\end{abstracts}

\maketitle

\selectlanguage{english}

\section{Introduction}
Holomorphic vector bundles on complex projective manifolds, and more generally on compact \ka\ manifolds, have been extensively studied. Here we consider holomorphic orbifold vector bundles on compact complex orbifolds. One can easily transfer the analytic methods from the manifold case and define \ka\ orbifold structures; also, notions like stability of vector bundles carry over. We first observe that the Donaldson-Uhlenbeck-Yau theorem generalizes to vector bundles on compact \ka\ orbifolds. So the moduli space of stable vector bundles with fixed determinant bundle on a compact complex orbifold admits a generalized \wp\ form. At this point the moduli space is considered as a reduced complex space with a \ka\ structure. Locally, on smooth ambient spaces, the generalized \wp\ form possesses differentiable $\pt\ol\pt$-potentials.

We call a \ka\ orbifold to be of {\em Hodge type}, if the \ka\ form is the Chern form of a positive orbifold line bundle. A theorem of Baily \cite{bai} states that in this case the underlying normal complex space is projective -- a fact that also follows from Grauert's theorem \cite[Satz 3]{grau}. For a \ka\ orbifold of Hodge type we see that there exists a determinant line bundle on the moduli space, equipped with a Quillen type metric, whose curvature form is equal to the generalized \wp\ form.

Orbifolds are examples of {\em analytic stacks}. For {\em algebraic stacks}, Lieblich's theorem, \cite{lieb}, says that the moduli space of stable vector bundles is in fact an {\em algebraic space}. In particular, it possesses a compactification as an analytic space. Our main result states that for moduli of stable holomorphic vector bundles the \wp\ form extends as a positive, closed current to a compactification, thus answering a conjecture of Andrei Teleman  (\cite{tele}) in the affirmative way. The extension property holds, whenever a given family of stable vector bundles on a \ka\ manifold is the restriction of a coherent sheaf.

For projective manifolds the \wp\ form is the curvature form of a certain determinant line bundle, equipped with a Quillen metric. This fact can be shown in the orbifold case by making use of a generalized Riemann-Roch formula for hermitian vector bundles on \ka\ orbifolds by X.~Ma.

{\bf Acknowledgements.} The second named author would like to thank Jean-Michel Bismut for his letter concerning the generalization of Theorem~\ref{th:bgs} and Xiaonan Ma for his letter and for sending \cite{ma}. The first named author acknowledges support by a J.~C.~Bose Fellowship.

\section{Complex orbifolds}\label{se:orbi}
Satake introduced orbifolds in \cite{satake}; he called them {\em V-manifolds}. By definition, the underlying topological space of an orbifold satisfies the second countability axiom (or ``the topology is countable at infinity'') together with an open locally finite covering $\{U_i\}_{i\in I}$; the local models of a complex orbifold of dimension $n$ are of the type $(W,\Gamma)$, where $W\subset \C^n$ is some domain, and $\Gamma$ is a finite group acting on $W$ in a
holomorphic way. By Cartan's theorem, one can assume that the action of $\Gamma$ on $W$ is linear. Furthermore, we assume that for a general point of $W$ the isotropy group is trivial.

Let $X$ be the underlying complex space of a complex orbifold $\underline{X}$ of dimension $n$. An {\em orbifold chart} or {\em local uniformizing system} on an open subset $$U\subset X$$ is a triple $(W, \Gamma, \varphi)$, where $(W,  \Gamma)$ is as above, and
$$
\varphi:W \to  U
$$
is a continuous map which induces a homeomorphism $W/\Gamma \to U$. Since quotients by cyclic groups that are generated by generalized reflections are smooth, sometimes it is assumed that all fixed point sets of non-trivial group elements are of codimension $\geq 2$. This assumption is not necessary in general, and we do not make this assumption at this point.

Let $(W', \Gamma', \varphi')$ be another orbifold chart for $U'\subset X$ such that $U\subset U'$. Then an {\em injection}
\begin{equation}\label{f1}
   (W,\Gamma, \varphi)\hookrightarrow  (W', \Gamma',\varphi')
\end{equation}
is given by a  biholomorphic map $\lambda$ from $W$ to an open subset of $W'$ satisfying the condition that for all $\gamma \in \Gamma$, there exists some $\gamma' \in \Gamma'$ with $\lambda \circ \gamma = \gamma' \circ \lambda$.

A collection $\mathcal{A}= \{(W_i, \Gamma_i, \varphi_i)\}$ of orbifold charts is called an {\em orbifold atlas} on $\underline{X}$, if the following conditions are satisfied:
\begin{enumerate}
  \item[(i)] The collection $\{U_i\}$ is an open cover, closed under finite intersections.
  \item[(ii)] For any $U_i \subset U_j$, there exist injections of the corresponding orbifold charts given by
\begin{equation}\label{lij}
\lambda_{ji}: W_i \to  W_j ,
\end{equation}
which are unique up to elements $g_i\in \Gamma_i$ in the sense that all such charts are of the form $g_{ji}\circ \lambda_{ji}$ for a unique $g_{ji}\in\Gamma_j$ and that for any $g_i\in \Gamma_i$, the map $\lambda_{ji}\circ g_i$ defines another orbifold chart.
 \item[(iii)] The composition of injections of orbifold charts is again an injection of orbifold charts.
\end{enumerate}
These axioms imply the properties below:
\begin{enumerate}
 \item[(1)] For any injection of orbifold charts $(W_i, \Gamma_i, \varphi_i) \to (W_j, \Gamma_j, \varphi_j)$, there exists a group homomorphism $\gamma_{ji}: \Gamma_i\to\Gamma_j$ such that
$$
\lambda_{ji}\circ g_i = \gamma_{ji}(g_i)\circ \lambda_{ji}
$$
for all $g_i\in \Gamma_i$, where $\lambda_{ji}$ is as in \eqref{lij}.
\item[(2)] For any $U_i\subset U_j \subset U_k$ and corresponding orbifold charts, we have
     $$
     \lambda_{kj}\circ \lambda_{ji} = g_{kji}\circ\lambda_{ki}
     $$
     for a unique $g_{kji}\in \Gamma_k$, where $\lambda_{kj}$, $\lambda_{ji}$ and $\lambda_{ki}$ are as in \eqref{lij}.
\item[(3)] With the above notation, $\gamma_{ki}= Ad(g_{kji})(\gamma_{kj}\circ\gamma_{ji})$.
\end{enumerate}
It may be mentioned that the group action is not required to be effective, like in certain moduli theoretic applications. The existence of such a homomorphism $\Gamma\to \Gamma'$ is to be considered as a part of the defining data.

Equivalence of atlases is defined in the usual way. The corresponding open covers will always be locally finite.

From now on we will fix a compact complex orbifold $\underline X = (X, \cA)$ of dimension $n$, where $X$ is the underlying space and $\cA$ is an equivalence class of atlases.

Given an orbifold chart $(W, \Gamma, \varphi)$, the group $\Gamma$ acts in an equivariant way on the complex tangent bundle $TW$. Given an embedding of orbifold charts $\lambda:W \to W'$ as in \eqref{f1}, the holomorphic map $\lambda$ lifts to a homomorphism of the holomorphic tangent bundles which is compatible with the homomorphism $\Gamma \to \Gamma'$ of finite groups.

The above observation gives rise to the definition of a holomorphic orbifold bundle on a complex orbifold $\ul X$.
\begin{definition}\label{de-bs}
   Let $\underline{X}$ be a complex orbifold with a system of orbifold charts $(W_i,\Gamma_i, \varphi_i)_{i\in I}$. A holomorphic orbifold vector bundle on $\underline{X}$ is defined by a system $E_i$ of\/ $\Gamma_i$-equivariant holomorphic vector bundles on $W_i$. Furthermore for every embedding
$$
(W_i,\Gamma_i, \varphi_i) \hookrightarrow  (W_j,\Gamma_j, \varphi_j)
$$
with holomorphic maps $\lambda_{ji}:W_i\to W_j$, there is a holomorphic isomorphism $E_i \to \lambda^*_{ji} E_j$ compatible with the equivariant actions of $\Gamma_i$ and $\Gamma_j$ and the group homomorphism $\Gamma_i \to \Gamma_j$.

Orbifold sheaves are defined in an analogous way.
\end{definition}

If $(W,\Gamma,\varphi)$ is an orbifold chart, and $E$ is a $\Gamma$-equivariant holomorphic vector bundle on $W$ of rank $r$, then $\Gamma$-invariant holomorphic sections of $W$ are known to give rise to a coherent torsion free analytic sheaf on the quotient $W/\Gamma$.
By our assumption on the isotropy groups of general points, the rank of the sheaf of invariant sections may be strictly smaller than the rank of $E$. This can happen, if the action of $\Gamma$ on $W$ is trivial but not on $E$. (There exist obvious trivial examples of this kind).

Note that notions like ``orbifold holomorphic differential forms'' or ``vector fields'' always refer to objects (and sheaves) on $\ul X$ that are defined in terms of holomorphic orbifold vector bundles.

Starting from the equivariant index theorem of Atiyah and Singer, in \cite{kawasakiRR} (cf.\ also \cite{kawasakiST}), Kawasaki proved the Riemann-Roch-Hirzebruch theorem for holomorphic V-vector bundles on V-manifolds.

By definition, a holomorphic section of a holomorphic orbifold vector bundle over an open subset of the underlying space $X$ is given in terms of orbifold charts $(W_j,\Gamma_j, \varphi_j)$ as $\Gamma_j$-invariant holomorphic sections on the spaces $W_j$ compatible with the gluing data. In a similar way differentiable sections of orbifold vector bundles (in the differentiable category) are defined. The orbifold tangent bundle, and its complexification together with the type decomposition of its exterior products yield further examples. Accordingly orbifold differential forms (holomorphic or differentiable) are defined with respect to orbifold charts $(W_j,\Gamma_j,\varphi_j)$ as $\Gamma_j$-invariant forms. The sheaf of (invariant) holomorphic orbifold $p$-forms defines a coherent sheaf $\Omega^p_{\ul X}$ on $X$. In a similar way the sheaf $End(\ul E)$ of orbifold endomorphisms is defined.

If $E$ is a holomorphic orbifold vector bundle, then $\Omega^p_{\ul X}(E)$ denotes the coherent sheaf on the underlying space $X$ that is defined by $E$-valued holomorphic $p$-forms on orbifold charts, equivariant under the respective group actions on orbifold charts. In particular $\cO_{\!\ul X}(E)$ denotes a coherent $\cO_X$-module.

From now on, we assume that the actions of the groups $\Gamma_j$ are {\em effective} so that the rank of the orbifold vector bundle $E$ on $\ul X$ is equal to the rank of the coherent sheaf $\cO_{\ul X}(E)$  -- such orbifold vector bundles are also called proper. Also the orbifold bundle $\det(E)=\Lambda^r(E)$ for $r=\text{rk}(E)$ defines a coherent sheaf on $X$. Let
$$
0 \to \Omega^p_{\ul X}(E) \to \cA^{p,0}_{\ul X}(E) \to \ldots \to\cA^{p,n}_{\ul X}(E)   \to 0
$$
be the orbifold Dolbeault complex of sheaves on $X$. It defines a fine resolution of the coherent sheaf $\Omega^p_{\ul X}(E)$ on $X$. We will denote the coherent sheaf $\cO_{\ul X}(E)$ on $X$ also simply by the same letter $E$, if there is no confusion.

\begin{definition}
A {\em holomorphic family} of holomorphic orbifold bundles over a complex space $S$ is an orbifold vector bundle $\cE$ over the orbifold $\ul X \times S$ (with orbifold structure induced by $\ul X$). The fibers $\cE_s$ for $s\in S$ are the restrictions of $\cE$ to $\ul X \times \{s\}$. A {\em deformation} of an orbifold vector bundle $E$ on $\ul X$ over a parameter space $S$ with base point $s_0\in S$ consists of a holomorphic orbifold vector bundle $\cE$ over $\ul X \times S$ together with an isomorphism $E \stackrel{\sim}{\longrightarrow} \cE_{s_0}$. Again, if no confusion is possible, we will denote the sheaf of holomorphic sections of $\cE$ by the same letter.
\end{definition}

In a similar way, the de Rham complex is defined on orbifolds. Since the ``constant orbifold sheaf'' $\R$ descends to $X$ as the constant sheaf $\R$ on the underlying normal space $X$, the orbifold de Rham cohomology is identified with the cohomology of the constant sheaf $\mathbb R$ on the underlying space $X$.

Hermitian metrics on holomorphic orbifold bundles $E$ are by definition hermitian structures of class $\cinf$ on the respective uniformizing systems $(W, \Gamma, \varphi)$ that are invariant under the action of the group $\Gamma$ and compatible with transition maps of orbifold charts in the sense of Definition~\ref{de-bs}. Even if the sheaf $\cO_\ul X(E)$ of invariant sections of an orbifold bundle $E$ is locally free, the induced hermitian structure on $\cO_\ul X(E)$ need only be continuous. Orbifold hermitian metrics can be constructed from invariant metrics on the orbifold charts by means of differentiable cut-off functions (after replacing the underlying space by an open relatively compact subspace, if necessary).

Since both the de Rham and Dolbeault complexes, as well as type decompositions of orbifold differential forms, are meaningful, there is the notion of a {\em \ka\ orbifold}. We denote a \ka\ orbifold form by
$$
\omega_X=\frac{\ii}{2} g_{\alpha\ol\beta}(z)\; dz^\alpha\wedge dz^\ol\beta,
$$
where $(z^1,\ldots,z^n)$ are holomorphic coordinates on a local chart $W$. Observe that local $\pt\ol\pt$-potentials for \ka\ forms on orbifolds descend to the underlying space as continuous functions. The theory of elliptic operators on orbifolds was developed in \cite{kawasakiRR}.

\section{Stable and Hermite-Einstein bundles on orbifolds}
Integration of orbifold differential forms is defined as follows: For any local uniformizing system $(W,\Gamma,\varphi)$ with $U := W/\Gamma$, and for any orbifold differential form $\eta$ of top degree on $\underline X$, define the integral
\begin{equation}\label{eq:locint}
   \int_U \eta := \frac{1}{m}\int_W \eta,
\end{equation}
where $m$ is the order of $\Gamma$. In order to define the integral of an orbifold form over the whole space $X$, a partition of unity on the underlying space $X$ with respect to the covering $\{U_i\}$ is being used together with the above formula \eqref{eq:locint}.

The Chern classes of a holomorphic orbifold bundle $E$ are defined in terms of hermitian (orbifold) metrics $h$ following Chern-Weil theory. Also the first Chern class of the determinant bundle of a torsion-free orbifold sheaf $\det(\cE)$ is well-defined. In case it only exists as a line bundle after taking a suitable tensor power, the Chern class may only
be defined over the rationals (and not over the integers).  The degree $\deg(\cE)$ with respect to a \ka\ form $\omega_X$ can be defined in the usual way by
$$
\deg(E) := \int_X c_1(\det(\cE)){\omega^{n-1}_X},
$$
where $n$ is the complex dimension of $X$. (We set $\omega^{k}_X=\omega_X\wedge\ldots \wedge \omega_X/k!$ for any $k$.)

Henceforth, we assume that the complex orbifold $\underline X$ is equipped with a \ka\ form $\omega_X$.

Since the slope (i.e.\ the quotient ${\rm degree}/{\rm rank}$) of a torsion-free coherent analytic sheaf is now defined, it follows immediately that $\mu$-stable and $\mu$-semistable sheaves are well-defined. Again stable bundles are simple; the space of holomorphic endomorphisms is defined in terms of the orbifold structure, meaning endomorphisms, by definition, commute with the action of the groups.

Sobolev spaces for orbifolds exist and are defined in terms of the local uniformizing systems, and the methods of Donaldson \cite{do} and Uhlenbeck-Yau \cite{uhy} are applicable: {\em Any stable vector bundle on a compact \ka\ orbifold possesses a unique orbifold \he\ connection}.

The classical construction of universal deformations of simple vector bundles on complex manifolds (cf.\ \cite{forster}) can be carried over literally. If $S$ is a parameter space, the trivial orbifold structure on $S$ provides $\ul X\times S$ with a natural orbifold structure.

\begin{remark}\label{pr:fambdl}
{\rm For a reduced complex space $S$, a holomorphic family of orbifold bundles $\{\cE_s\}_{s\in S}$ parameterized by $S$ is given by a holomorphic orbifold bundle $\cE$ on the orbifold $\ul X\times S$ such that $\cE_s=\cE\vert_{\ul X\times \{s\}}$.}
\end{remark}

Using the implicit function theorem for Sobolev spaces one can show that given a holomorphic family of stable vector bundles on a compact \ka\ orbifold, the irreducible \he\ connection extends uniquely from a fiber to a differentiable family of \he\ connections.

Later we will consider moduli of orbifold vector bundles $E$ with a {\em fixed orbifold determinant bundle} $\det(E)$ on $\ul X$ (whose sections define a coherent sheaf on the underlying space $X$).

We call a \ka\ orbifold $(\ul X,  \omega_X)$ to be of {\em Hodge type}, if it possesses a hermitian holomorphic orbifold line bundle $(L,k)$, whose (orbifold) Chern form is a positive multiple of the \ka\ form $\omega_X$. According to Baily's theorem \cite{bai}, invariant sections of a power of the line bundle $L$ provide a projective embedding of the underlying normal space $X$.

\section{Positive line bundles on moduli spaces of stable bundles}
By a ``moduli space'' $\cM$ we will always denote a moduli space of stable orbifold vector bundles of {\em fixed} orbifold determinant line bundle $L$ (or more precisely a connected component). However, most of the following statements concerning holomorphic families are valid without this assumption, which will be needed in the global statement of Proposition~\ref{pr:fibintendo}.

For any holomorphic family of holomorphic orbifold vector bundles $\cE_s$ of rank $r$ there exists a number $m\in \mathbb N$ such that $(\det(\cE_s))^m$ is an invertible $\cO_X$-module. Hence such a number exists uniformly for any component of the moduli space $\wt\cM$ of stable orbifold bundles, where the determinant need not be fixed. This fact implies the existence of a finite surjective map
$$
Pic^0(X) \times \cM \to \wt \cM,
$$
which motivates the extra condition. When dealing with the \wp\ current, we will assume that the {\em (fixed) determinant of the given holomorphic orbifold bundles} representing points of $\cM$ is an {\em invertible $\cO_X$-module}.

\subsection{Determinant line bundle and Quillen metric}\label{se:quill}
In this section we will introduce a natural \ka\ structure on the moduli space of stable holomorphic orbifold vector bundles using a certain determinant line bundle equipped with a Quillen metric. This \ka\ form will be called the \textit{generalized Weil-Petersson form}. We will see that its construction is functorial with respect to the base change of families.

Let $\{\cE_s\}_{s\in S}$ be a holomorphic family of holomorphic orbifold vector bundles on $\ul X$ in the sense of Remark~\ref{pr:fambdl} with fixed determinant orbifold line bundle. Let $\{H_s\}$ be a $C^\infty$ family of (orbifold) \he\ metrics. By definition, we have a hermitian metric $H$ on the complex orbifold vector bundle $\cE$ which is invariant under the actions of the groups. Let $F=F(\cE,H)$ be the curvature form of the unique hermitian connection on $\cE$ compatible with the holomorphic structure, and set
$$
\cK(\cE,H)= \frac{\ii}{2\pi} F(\cE, H).
$$
It is a $d$-closed, real orbifold $(1,1)$-form with values in the orbifold vector bundle $\text{End}^0(\cE)\subset \text{End}(\cE)$, since the determinant is being fixed. The \he\ condition reads
$$
\Lambda(\ii F|\cE_s) = \lambda \cdot {\rm id}_{\cE_s},
$$
where $\Lambda$ is the dual of the multiplication with $\omega_X$ fiberwise.

Given any tangent vector $v\in T_s S$, we denote its lift as vector field on $\ul X\times S$ along \hbox{$\ul X\times\{s\}$} by the same letter $v$. For a manifold $X$, the contraction $v \;\Big\lrcorner\; \cK(\cE,H)\vert_{\ul X \times \{s\}}$ is a $\ol\pt$-closed $(0,1)$-form on $X$ with values in ${End}(\cE_s)$ that represents the \ks\ class of $v$ \cite{s-toma}. In the orbifold case, it follows that we get an orbifold form in this way, which fits into the resolution
$$
\textit{End}(\cE_s)\to\cA^{0,\bullet}_\ul X(\textit{End}(\cE_s))
$$
of certain sheaves on the underlying space $X$.

The classical computation of infinitesimal deformations applies literally, and we see that the \ks\ map takes values in the first orbifold cohomology $H^1(\ul X,\textit{End}(\cE_s))$. We denote the \ks\ map by
\begin{equation}\label{eq:ks}
   \rho: T_sS \to H^1(\ul X,\textit{End}(\cE_s)) .
\end{equation}
It should be emphasized that the orbifold cohomology $H^1(\ul X,\text{End}(\cE_s))$ in \eqref{eq:ks} is the cohomology of the orbifold endomorphism bundle on $X$ which we compute by means of the orbifold Dolbeault complex.

\he\ metrics on the vector bundles $\cE_s$ together with the orbifold \ka\ structure on $\ul X$  induce natural inner products on the sections of $\cA^{p,q}_\ul X(\textit{End}(\cE_s))$. In \cite{kawasakiRR, kawasakiST} the corresponding theory of elliptic operators and harmonic sections was developed. The inner product of harmonic representatives of \ks\ classes induce a natural inner product on $T_sS$, which is called generalized \wp\ metric. We recall details.

Let $v\in T_sS$ be a tangent vector. Let $A\in \cA^{0,1}_\ul X(\textit{End}(\cE_s))(\ul X)$ be the harmonic representative of $\rho(v)$.

The generalized \wp\ norm of $v$ is given by
$$
\| v\|^2_{WP}:= \int_X \vert A\vert^2(z) g dV,
$$
where $gdV$ is the \ka\ volume form.

Like in {\cite[Proposition 1]{s-toma}}, we have
\begin{equation}\label{eq:harm}
   A=2\pi\; v \;\Big\lrcorner \;\cK(\cE,H)\vert_{\ul X\times\{s\}} .
\end{equation}

The proof of the following proposition holds in the orbifold case.
\begin{proposition}[{cf.\ \cite[Proposition~1]{s-toma}\label{pr:fibint}}]
The \wp\ form satisfies the following fiber integral formula:
\begin{gather}\label{eq:fibint}
   \frac{1}{2\pi^2}\omega^{WP}= - \int_{(X\times S)/S} {\rm tr} \left(\cK(\cE,H)\wedge \cK(\cE,H)\right)\wedge{\omega^{n-1}} \\ \nonumber  \hspace{5cm} + \frac{\lambda}{\pi}\int_{(X\times S)/S} {\rm tr}( \cK(\cE,H)) \wedge \omega^{n},
\end{gather}
where $\omega$ denotes the pull-back of the \ka\ form $\omega_X$ to $\ul X\times S$, and
$$
\lambda= \frac{2 \pi c_1(\cE_s)[X]}{r\cdot{\rm vol}(X)}
$$
(independent of $s$).

After applying a conformal factor that only depends upon the parameter $s\in S$ to the family $H$ of \he\ metrics we have (locally with respect to $S$)
\begin{equation}\label{eq:fibint_simpl}
   \frac{1}{2\pi^2}\omega^{WP}= - \int_{X\times S/S} {\rm tr} \left(\cK(\cE,H)\wedge \cK(\cE,H)\right)\wedge{\omega^{n-1}}.
\end{equation}

\end{proposition}

\begin{proof}
We may assume that $S$ is smooth. In terms of local coordinates $z^\alpha$ on $X$ and $s^i$ on $S$, we have
\begin{gather*}
{\rm tr}(-\cK(\cE,H)\we\cK(\cE,H))= \nonumber \hspace{5cm}\\ \frac{1}{2\pi^2} {\rm tr} (-F_{\alpha\ol\beta}F_{i\ol\jmath}+ F_{i\ol\beta}F_{\alpha\ol\jmath})\;\ii dz^\alpha\we dz^{\ol\beta}\; \ii ds^i\we ds^{\ol\jmath}.
\end{gather*}
so that
\begin{gather*}
  \int_{X\times S/S}{\rm tr}(-\cK(\cE,H)\we\cK(\cE,H))\we\omega^{n-1} = \hspace{5cm}\\
  \frac{1}{2\pi^2} \omega^{WP} - \frac{\lambda}{2\pi^2} \int_{X\times S/S} {\rm tr}(F_{i\ol\jmath})\omega^n \;\ii ds^i\we ds^{\ol\jmath}.
\end{gather*}
Now
$$
g^{\ol\beta\alpha}F_{i\ol\jmath;\ol\beta\alpha} = g^{\ol\beta\alpha}F_{i\ol\beta;\ol\jmath\alpha}=
g^{\ol\beta\alpha}(F_{\alpha\ol\beta;i \ol\jmath} + [F_{i\ol\beta},F_{\alpha,\ol\jmath}])
$$
implying that for all $s\in S$ the fiberwise Laplacian of the trace vanishes:
$$
\Box_s {\rm tr}( F_{i\ol\jmath})=0.
$$
So ${\rm tr} (F_{i\ol \jmath})$ only depends on $s\in S$, and it possesses a (local) $\pt\ol\pt$-potential $u$. We replace the family of hermitian metrics $H$ by $e^{-u/r}H$, where $r$ is the rank of $\cE$. This affects only the last term in the fiber integral formula \eqref{eq:fibint} for the \wp\ metric, which now vanishes.
\end{proof}

Again the proof is literally the same in the orbifold case.

It follows immediately from the definition \eqref{eq:locint} that $\omega^{WP}$ is of class $\cinf$, when $S$ is smooth. Note that both sides of \eqref{eq:fibint_simpl} are defined on the Zariski tangent spaces of $S$ and the equality holds there.
\begin{proposition}\label{propi1}
   The generalized \wp\ form is a $d$-closed real $(1,1)$-form on the parameter space $S$.

   It is strictly positive on any complex tangent space $T_sS$.

   When $S$ is reduced but possibly singular, it possesses a local $\pt\ol\pt$-potential of class $C^\infty$.

   The construction of the generalized \wp\ form is compatible with base change: For $\psi : W\to S$, the \wp\ form for the pull back to $X\times W$ of the given family of vector bundles coincides with the form $\psi^* \omega^{WP}$.
\end{proposition}

\begin{proof}
Using a partition of unity, the integral is evaluated locally on the domains $W\times S$ of orbifold charts $(W\times S, \Gamma, \varphi \times id_S)$ (and the results are divided by the order of the groups $\Gamma$). This fact implies the differentiability of the result of the fiber integration. The \ka\ property of the \wp\ form follows from \eqref{eq:fibint_simpl}. Indeed, the $d$-closedness of the real $(1,1)$-form $\omega^{WP}$ is a consequence of the closedness of the integrand, which is evidently a $d$-closed real $(n+1,n+1)$-form.

The positivity of $\omega^{WP}$ on a Zariski tangent space of the base follows from its definition.

The existence of a local potential in the case of a singular, reduced parameter space follows from \cite[Theorem 10.1 and \S~12]{f-s}.

Since the \he\ connection on a stable vector bundle is unique, the assignment of a harmonic representative of a \ks\ class to a tangent vector is functorial. Also the integrand of the fiber integral in \eqref{eq:fibint_simpl} is compatible with the base change of homomorphisms.
\end{proof}

Equation \eqref{eq:harm} can be interpreted as the first variation of the \he\ form in a holomorphic family.

We state Proposition~\ref{pr:fibint} in terms of the Chern character form of the endomorphism bundle. The induced metric on $End(\cE)$ is denoted by the same letter $H$.
\begin{proposition}\label{pr:fibintendo}
\begin{gather}
\frac{1}{2\pi^2}\omega^{WP}= -\frac{1}{r}\int_{X\times S/S} ch_2(End(\cE),H)\we \omega^{n-1} \nonumber\hspace{4cm}\\
\hspace{1cm}+\frac{1}{r \pi^2}\left(\int_{X\times S/S} {\rm tr}(F_{i\ol\beta}){\rm tr}(F_{\alpha\ol\jmath})g^{\ol\beta\alpha} \omega^n \right) \ii ds^i\we ds^\ol\jmath.\label{eq:fibintendo}
\end{gather}
If the determinant bundles $\Lambda^r(\cE_s)$ are constant, then
\begin{equation}\label{eq:fibintendosimp}
\frac{1}{2\pi^2}\omega^{WP}= -\frac{1}{r}\int_{X\times S/S} ch_2(End(\cE),H)\we \omega^{n-1}.
\end{equation}
\end{proposition}
\begin{proof}
We apply the formula $ch_2(End(\cE),H)= 2r\cdot ch_2(\cE,H) - c^2_1(\cE,H)$ (for any hermitian metric), and the fact that the Chern character form has the representation
$$
ch(\cE,H) = {\rm tr}(\exp(\cK(\cE,H))).
$$
Observe that the forms ${\rm tr}( F_{i\ol\beta})dz^\ol\beta$ are harmonic. They represent the \ks\ classes for the family $\{\Lambda^r(\cE_s)\}$ of line bundles. If the latter is constant, then the ${\rm tr}(F_{i\ol\beta})$ vanish identically.
\end{proof}

\subsection{Relative Riemann-Roch Theorem for hermitian vector bundles over \ka\ manifolds}
Our result depends heavily upon the theorem of Bismut, Gilet and Soul\'e. Given a proper, smooth \ka\ morphism $f:\cZ \to S$ with relative \ka\ form $\omega_{\cZ/S}$ and a hermitian line bundle $(\cF,h)$ on $\cZ$, there exists a Quillen metric $h^Q$ on the determinant line bundle
$$
\lambda=\lambda(\cF):= {\rm det}f_!(\cF)
$$
taken in the derived category satisfying

\begin{theorem}[\cite{bgs}]\label{th:bgs}
The Chern form of the determinant line bundle $\lambda(\cF)$ on the base $S$ is equal to the component in degree two of the following fiber integral.
   \begin{equation}\label{eq:bgs}
      c_1(\lambda(\cF),h^Q)= -\left[\int_{\cZ/S}\textit{td} (\cZ/S,\omega_{\cZ/S})\textit{ch}(\cF,h)\right]^{(2)}
   \end{equation}
 Here $\textit{ch}$ and $\textit{td}$ resp.\ stand for the Chern and Todd character forms resp.
\end{theorem}

\subsection{Relative Riemann-Roch Theorem for hermitian orbifold vector bundles on orbifolds}
For a finite group $G$, acting holomorphically and effectively on a compact manifold $Z$ and an equivariant holomorphic vector bundle $\ul E$, in \cite[formula (4.4)]{a-s} Atiyah and Singer proved the equivariant index theorem. If $E$ denotes the coherent sheaf of invariant sections of $\ul E$ on the quotient space $Z/G$ the result is
\begin{equation}
   \chi(Z/G,E)= \frac{1}{|G|} \sum_{\gamma\in G}\langle \cI^\gamma(Z;\ul E)\rangle[Z_\gamma]
\end{equation}
where $Z_\gamma$ denotes the fixed point set of the automorphism $\gamma$ and $\cI^\gamma(Z;\ul E)$ the equivariant Todd class (which is evaluated over $Z_\gamma$ for all group elements including the identity). In particular, for the Dolbeault operator in the equivariant case, they computed the Euler-Poincar\'e characteristic of the sheaf of invariant holomorphic sections on the quotient \cite[Theorem (4.7)]{a-s}. Here one has to evaluate on all fixed point sets of the elements of $G$ characteristic classes for the restrictions of the given vector bundle and the normal bundle of the fixed point sets.

The relative setting of \eqref{eq:bgs} from \cite{bgs} was solved for holomorphic immersions by Bismut and Lebeau in \cite{b-l}. This theorem was later on expanded to the more involved equivariant case of immersions and submersions by Ma and Bismut in \cite{bi-ma, ma_suequi}. The orbifold case was treated as a local version of the equivariant situation. In \cite{kawasakiRR}, Kawasaki gave a heat equation proof of the Riemann-Roch Theorem for orbifolds.

The generalization of Theorem~\ref{th:quil} was shown by X.~Ma in \cite{ma}. The new ingredient of the Riemann-Roch formula is the {\em associated singular orbifold structure}.

Let a complex orbifold $(\ul X, \cA)$ be given in the sense of Section~\ref{se:orbi}. For any orbifold chart $(W,\Gamma, \varphi)$ we already assumed that $\Gamma$ acts {\em effectively} on $W$.  A point $w \in W$ is called {\em singular}, if the isotropy group $\Gamma_w$ is non-trivial, otherwise the point is called {\em regular}. The sets of singular points give rise to a disjoint union of orbifolds $\Sigma \ul X$, which are immersed in $\ul X$ as the singular stratification (cf.\ \cite{kawasakiST}).

The local models are defined as follows adopting the notation of Section~\ref{se:orbi}. Let $x$ be a singular point and $W_x$ a neighborhood, which is invariant under the action of $\Gamma_x$. We now have an orbifold chart $(W_x, \Gamma_x, \varphi_x)$, where $\varphi_x: W_x \to U_x$ is the induced quotient map. For all elements $h^j_x \in \Gamma_x$ we have the conjugacy classes $(1), (h^1_x), \ldots, (h^{\rho_x}_x)$ in $\Gamma_x$. Observe that the conjugacy classes of elements correspond to the $\Gamma_x$-orbits of the fixed point sets. Up to isomorphism, these only depend on the image of the point $x$ in $U_x \subset X$, so that $x$ can be considered as a point in $X$.

Denote the sets of fixed points by $W^{h^j_x}_x$. These sets  are equipped with a natural action of the centralizers $Z_{\Gamma_x}(h^j_x)$. Since this action need not be effective, the centralizer can be divided by the ineffectivity kernel $K^j_x$. The number $m_j$ of the elements of $K^j_x$ is called the {\em multiplicity} of the orbifold $\Sigma{\ul X}$ at $(x, h^j_x)$. The resulting group acting on $W^{h^j_x}_x$ is $Z_{\Gamma_x}(h^j_x)/K^j_x$. The orbifold chart $(W_x,\Gamma_x,\varphi_x)$ gives rise to a union of orbifold charts. The disjoint union of different quotients $W^{h^j_x}_x/Z_{\Gamma_x}(h^j_x)$ corresponds to the set $\{(y,(h^j_y)); y \in W_y, j=1,\ldots, \rho_y \}$. (The index $j=0$ with $h^0_y=1$ that yields the total orbifold $\ul X$ is being excluded).

These orbifold charts fit together and define the singular orbifold structure $\Sigma\ul X$, namely
$$
\Sigma\ul X= \{ (x, (h^j_x)); x \in X, \Gamma_x \neq 1, j=1,\ldots,\rho_x\}.
$$
The images of the singular components in the underlying space $X$ are denoted by $X_j$.

Our approach relies on the following theorem of Ma, which is stated for the general situation of proper holomorphic orbifold submersions. In our case we are looking at an orbifold $\ul X$ together with the canonical projection $\pi : \ul X\times S \to S$. We state the theorem for this special case (for the extension from smooth parameter spaces $S$ to reduced complex spaces as parameter spaces cf.\ \cite[Appendix]{f-s}).

By assumption $\ul X$ carries a Kähler structure. Denote the hermitian metric on the orbifold tangent bundle by $g^{TX}$. Let $\xi$ be an orbifold vector bundle on $\ul X\times S$ equipped with a hermitian metric $h^\xi$. In \cite{ma} the holomorphic determinant bundle $\lambda$ is constructed and equipped with a Quillen metric $h^Q$. The induced hermitian connection is denoted by $\nabla^\lambda$.

\begin{theorem}[X.~Ma, Theorem 3.3 \cite{ma}]\label{th:ma}
   The curvature of the Quillen metric on the determinant line bundle $\lambda$ equals
   \begin{equation}\label{eq:ma}
      (\nabla^\lambda)^2 = 2\pi \sqrt{-1} \Big[\Sigma_{j\geq 0} \frac{1}{m_j} \int_{X_j} td^{\Sigma}(TX,g^{TX})ch(\xi,h^\xi)  \Big]^{(2)}.
   \end{equation}
\end{theorem}

Now we assume that the orbifold $\ul X$ is of Hodge type with $\omega_X=c_1(L,k)$.
\begin{theorem}\label{th:quil}
   Let $S$ be a reduced complex space. Then up to a numerical factor the generalized \wp\ form $\omega^{WP}$ for stable orbifold vector bundles $\cE$ on $\ul X\times S$ is equal to the Chern form of a natural line bundle $\lambda$ equipped with a hermitian metric $h^Q$:
   $$
      c_1(\lambda, h^Q)\simeq \omega^{WP}.
   $$
The construction of this holomorphic, hermitian line bundle is functorial in $S$: If $$\psi:W \to S$$ is a base change morphism, then the line bundle for the given family pulled back to $X\times W$ is $\psi^*\lambda$ equipped with the pulled back hermitian metric $\psi^*h^Q$.

Moreover, the hermitian line bundle $(\lambda, h^Q)$ descends to the moduli space as such.
\end{theorem}

\begin{proof}
We first note that also in the orbifold case the spectral determinant can be identified with the Knudsen-Mumford determinant \cite[Section 5]{ma}. By a purely formal argument, the Riemann-Roch formula \eqref{eq:ma} holds for elements of the Grothendieck group.

Define an element $\chi$ of $\oplus_{i\geq 0}\cA^{i,i}_{\ul X \times S}(\ul X\times S)$ by
\begin{equation}\label{eq:defchi}
   \chi= \left(ch(\textit{End}(\cE), H)-\cO^{r^2}_{\ul X\times S}) \right) \cdot ch\left(((L,k) - (L,k)^{-1})^{\otimes(n-1)}\right).
\end{equation}
Observe that the Chern character form of the hermitian orbifold bundle $\textit{End}(\cE)$ is defined in terms of the curvature form, which is an invariant form, so that $\chi$ is again given in an invariant way. \qed

All components of the form $\chi$ in degree smaller or equal to $2n$ vanish, and the term of smallest degree is the $(n+1,n+1)$-component
$$
\chi' = 2^{n-1} c_1(L,k)^{n-1}\cdot ch_2(End(\cE),H).
$$
Let $\pi:X\times S\to  S$ be the natural projection. Then the $(1, 1)$-component of the pushforward $\pi_*\chi$ is the $(1,1)$-component of the fiber integral in \eqref{eq:fibintendosimp}. Hence it coincides with $\omega^{WP}$ up to multiplication by a numerical constant.

Consider the following line bundle on $S$
\begin{equation}\label{lambda}
   \lambda:= \det \pi_!\left((\textit{End}(\cE) -\cO_{X\times S}^{r^2}) \otimes(\pi^*L- \pi^*L^{-1} )^{\otimes(n-1)} \right)^{-1}.
\end{equation}
Note that for proper smooth holomorphic maps of complex spaces, and for projections $\ul X \times S \to S$, where $\ul X$ denotes a compact complex orbifold, the direct image taken in the derived category can be computed in terms of Forster-Knorr systems \cite{f-k1}. These simplicial objects in the category of finite free coherent modules are used to represent the direct image locally up to quasi-isomorphism by {\em bounded} complexes of finite free modules so that a determinant line bundle exists.

For locally free sheaves $\cE'$ and $\cE''$, the determinant line bundle of $\cE'-\cE''$ is defined as
$$
\det\pi_!(\cE'-\cE'') = \det\pi_!(\cE') \otimes \det\pi_!(\cE'')^{-1} .
$$
Note that $\chi$ is the Chern character form of
\begin{equation}\label{eq:eF}
   (\textit{End}(\cE)-\cO_{X\times S}^{r^2}) \otimes(\pi^*L- \pi^*L^{-1} )^{\otimes(n-1)}
\end{equation}
with the induced hermitian metrics.

For notational convenience, we will denote the element in \eqref{eq:eF} by $\cF$. We equip $\cE$ with the fiberwise \he\ metric. Furthermore, $\omega_X$ defines a relative metric $\omega_{X\times S/S}$.

Now we apply Theorem~\ref{th:ma}: Since the integrand is of degree $(n+1,n+1)$, the contribution of the singular part $\Sigma \ul X$ vanishes identically so that
\begin{equation}\label{eq:bgsrr}
   c_1(\lambda,h^Q)= - \left(\int_{X\times S/S} {\rm td}(X\times S/S,\omega_{X\times S/S})\cdot {\rm ch}(\cF,h) \right)^{(1,1)}.
\end{equation}

Consider $\chi$ defined in \eqref{eq:defchi}. Note that it does not contain any term of degree $(j,j)$ for $j\leq n$. Since we are computing the $(1,1)$-component of a fiber integral, and the fibers are of complex dimension $n$, this implies that the only contribution of the Todd character form in \eqref{eq:bgsrr} is the constant one. Therefore the integrand in \eqref{eq:bgsrr} is equal to $\chi'$.

Although there need not be a universal bundle on the moduli space, the determinant line bundle $\lambda$ given in \eqref{lambda} is well-defined: Since the families $\cE$ on $X\times S$ are unique up to pull-backs of line bundles on $S$, the endomorphism bundles $End(\cE)$ and hence $\lambda$ exist globally.

Now the claim follows from \eqref{eq:fibint}.
\end{proof}

\section{The \wp\ current}\label{se:wpcurr}
The following theorem is a generalization of a result by Teleman (\cite[Theorem 1.4]{tele}) and answers a conjecture of his (Conjecture~2 in p.\ 544 of \cite{tele}) in an affirmative way for projective varieties.

We note that our proof also applies to families of vector bundles on a Kähler manifold that {\em degenerate to coherent sheaves}. We cannot make any statement about conceivable transcendental degenerations.

By definition a positive closed current on a reduced complex space is a current on the normalization. We only need the weakest extension property for the \wp\ form, namely the existence of an extension to a desingularization of a compactification.

\begin{theorem}\label{th:ext}
   Let $\cM$ be the moduli space of stable holomorphic vector bundles on a projective manifold $X$ such that the determinant line bundles are fixed. Then there exists a compactification $\ol\cM$ such that the pull-back of the generalized \wp\ form to a desingularization of $\ol\cM$ extends as a positive $d$-closed $(1,1)$-current.
\end{theorem}

The strategy of the proof of Theorem~\ref{th:ext} is as follows. We show that there are local extensions; this is proved in Theorem~\ref{pr:ext}. In view of this fact, the proof of Theorem~\ref{th:ext} is completed by Proposition~\ref{pr:extcurr}.

\begin{proposition}[cf.\ {\cite[Proposition~8]{sch-pos}}] \label{pr:extcurr}
   Let $A$ be a closed analytic subset of a complex manifold $Y$, and let $Y':= Y\setminus A$. Let $\omega'$ be a closed, positive current on $Y'$, whose Lelong numbers vanish everywhere. Assume that for any point of $A$ there exists an open neighborhood $U\subset Y$ such that $\omega'\vert_{Y'\cap U}$ possesses an extension to $U$ as a closed positive current. Then $\omega$ can be extended to all of $Y$.
\end{proposition}
\begin{proof}
We first assume that $A$ is contained in a simple normal crossings divisor. Let $\omega_U$ be a (positive) extension of
$\omega'|U\cap Y'$.

We apply Siu's decomposition theorem \cite{siu:curr}:
$$
\omega_U = \sum_{k=0}^\infty \mu_k [Z_k] + R\,
$$
where the $[Z_k]$ are currents of integration over irreducible analytic sets of codimension one, and R is a closed positive current with the property that for the sets $E_c(R)$ of points where the Lelong number is greater or equal to $c>0$ the dimension $\dim E_c(R) < \dim Y-1$ for every $c > 0$. The decomposition is locally and globally unique.

Since the Lelong numbers of $\omega'$ vanish everywhere, the sets $Z_k$ must be contained in $A$ so that the positive residual current $R$ is the null extension of $\omega'|U\cap Y'$. We now take the currents of the form $R$ as local extensions. If $W\subset Y$ is open, then the difference of any two such extensions is a current of order zero, which is supported on $A\cap W$. By \cite[Corollary III (2.14)]{demaillybook} it has to be a current of integration supported on $A\cap W$, so it must be equal to zero.

In the general case  we consider a desingularization of the pair $(Y,A)$, i.e.\ an embedded resolution,  and use the local construction. Since the local extensions possess locally  \psh\ potentials, pull-backs as positive currents exist. The Lelong numbers are still equal to zero, and the previous argument applies so that the pull-back of $\omega'$ extends as a positive current $\wt \omega$. The push-forward of $\wt\omega$ solves the problem.
\end{proof}

In view of Proposition~\ref{pr:extcurr} we only need to consider a {\em polydisk} $S$ as base space and assume that the analytic set $A\subset S$ is contained in the zero-set of a product of certain coordinate functions:

\begin{theorem}\label{pr:ext}
Let $\cE$ be a coherent sheaf on $X\times S$ which is $\cO_S$-flat and let  $\cE_s=\cE\vert_{X\times \{s\}}$ be stable and locally free for $s\in S\backslash A$. Then  $\omega^{WP}$ extends to $S$ as a  positive, closed  $(1,1)$-current.
\end{theorem}

\begin{remark}
   {\rm Such a base space suitably taken will dominate an open neighborhood of any boundary point of $\cM$ in a compactification $\ol{\cM}$.}
\end{remark}

\subsection{Construction of initial metrics}\label{se:initial}
We construct initial metrics in the situation given above in this section (in particular $S$ is a polydisk). All sheaves $\Lambda^{max}\cE_s$, $s \in S'$ are assumed to be isomorphic to an invertible sheaf, which we denote by $\cL_X$ on $X$. Let $p:X \times S \to S$ and $q:X\times S \to X$ be the canonical projections.

According to Hironaka's flattening theorem, we have a sequence of blow-ups with regular centers over $X\times A$ giving rise to a modification $\pi: \wt{X\times S}\to  X\times S$ so that
$$
\wt\cE=\pi^* \cE/\mathit{torsion}
$$
is locally free. Let $\wt\cL= \Lambda^{max}(\wt\cE)$ and consider the invertible sheaf $\wt\cL \otimes \pi^*q^* \cL^{-1}_X$ on $\wt{X\times S}$, whose restriction to the preimage of $X\times S'$ is of the form $\pi^*p^* \cL_{S'}$, where $\cL_{S'}$ is an invertible sheaf on $S'$, since all $\Lambda^{max}\cE_s$ are isomorphic for $s\in S'$. Now the direct image $\wh\cL=p_*\pi_*(\wt\cL \otimes \pi^*q^*\cL^{-1}_X)$ restricted to $S'$ equals $\cL_{S'}$. We use the flattening theorem again, and obtain a modification $\nu:\wt S\to S$, which is an isomorphism over  $S'$, such that $\wt\cL_{\wt S}= \nu^*(\wh\cL)/\mathit{torsion}$ is invertible.

We look at
$$
\xymatrix{\wt{X\times S}\times_S\wt S \ar[r]^{\mu} \ar[dd]^{\wt p}  &  \wt{X \times S} \ar[d]^{\pi} &\\
& X\times S \ar[r]^q \ar[d]^p  & X\\  \wt S  \ar[r]^\nu & S &
}
$$
with canonical projections $\wt p$ and $\mu$. Note that both $\wt{X\times S}$ and $\wt{X\times S}\times_S\wt S$ are smooth.

Now
$$
\cG=\Lambda^{max}(\mu^*\wt \cE) \otimes \mu^*\pi^*q^* \cL^{-1}_X \otimes \wt p^* \wt\cL^{-1}_{\wt S}
$$
is trivial over $\wt p^{-1}\nu^{-1} S'$. There exists a meromorphic section $\sigma$ of $\cG$ that has no zeroes on $\wt p^{-1}\nu^{-1} S'$. Then $1/|\sigma|^2$ defines a hermitian metric of $\cG|\wt S'$.

Let $\wt H_0$ be any hermitian metric on $\wt \cE$, and $k_X$, $k_{\wt S}$ resp.\ hermitian metrics on $\cL_X$ and $\wt\cL_{\wt S}$ resp.

{\em We assume that $\Lambda c_1(\cL_X, k_X)= {\rm const.}$, i.e.\ the curvature form of $k_X$ is harmonic.}

Then
$$
w= \mu^*\det(\wt H_0)\cdot \mu^*\pi^* q^* (k^{-1}_X)\cdot \wt p^* (k^{-1}_{\wt S})\cdot |\sigma|^2
$$
is a differentiable function. Now
\begin{equation}\label{eq:h0}
\wh H_0 = w^{-1/r} \cdot \mu^* \wt H_0
\end{equation}
defines a hermitian metric on $\mu^*\wt\cE$ over $\wt p^{-1}\nu^{-1}(S')$ so that
\begin{equation}\label{eq:deth0}
\det(\wh H_0) = \mu^*\pi^* q^*(k_X)\cdot \wt p^* (k_{\wt S})/ |\sigma|^2.
\end{equation}
Note that the factor $|\sigma|^2$ does not contribute to the curvature (over $\nu^{-1}(S')$). For all $s\in S'$ the curvature form of $\det(\wh H_0)$ is harmonic -- this fact will be needed when using the heat equation approach to \he\ metrics in the following section.

\subsection{Extension of $\omega^{WP}$}\label{sec:extomega}

Let ${\wt S}'= \nu^{-1}(S')$. We use the solution of the heat equation with the initial metric from the preceding section~\ref{se:initial} over ${\wt S}'$. We refer to the approach and exposition by Siu (\cite{siu}). The following proposition implies the statement of Theorem~\ref{pr:ext}.

\begin{proposition}\label{pr:extomega}
The \wp\ form $\omega^{WP}_{S'}$ possesses an extension to $S$ as a positive $d$-closed $(1,1)$-current.
\end{proposition}

In order to estimate the \wp\ form in degenerations we will use the fiber integral formula in the form of \eqref{eq:fibint}.

We consider the heat equation for an $s\in \wt S'$. For $0\leq t$ we have endomorphisms $h(t,s)$ of the bundles $\cE_s$ with $\det (h(t,s))=1$ and
\begin{equation}\label{eq:heat}
\frac{\pt h}{\pt t} h^{-1}= - (\Lambda_X F_s -\lambda id_{\cE_s}),
\end{equation}
where $\Lambda_X$ denotes the contraction with $\omega_X$.

We write \eqref{eq:fibint}  in the form
\begin{equation}\label{eq:wpspec}
\omega^{WP}_{S'}= - \frac{1}{2}\int_{X\times S'/S'} {\rm tr}(F\we F) \we \omega^{n-1} + \lambda \int_{X\times S'/S'} {\rm tr}(F)\we \omega^n,
\end{equation}
where $F$ denotes the curvature form of the fiberwise \he\ metric over the total space of the family, and $\omega = q^*\omega_X$ as above.

We consider the second term of \eqref{eq:wpspec}. Since $\det h(t,s) =1$ for all $t$ and $s$, we can replace ${\rm tr}(F)$ in this term by ${\rm tr}(F_0)$, where $F_0$ is the curvature of the initial metric. It follows immediately from the definition that the integral depends only on the horizontal components of ${\rm tr}(F_0)$, which involve the parameter $s$. Using \eqref{eq:deth0} we conclude that this contribution is equal to the curvature of $k_{\wt S}$ pushed forward to $S$ under $\nu$. We have
$$
-\omega_{\wt S,aux} \leq \ii \pt \ol\pt k_{\wt S} \leq \omega_{\wt S,aux}
$$
for some local \ka\ form $\omega_{\wt S,aux}$.

Consider the first term in \eqref{eq:wpspec}.
Over all of $X\times S'$ we have
$$
\frac{1}{2}\left({\rm tr} (F(t)\we F(t)) - {\rm tr} (F_0\we F_0) \right) = \ii\;\ol\pt\pt R_2(t),
$$
where
$$
R_2(t) = \int^t_0 {\rm tr}\big( F(\tau)\cdot \Lambda_X F(\tau)\big) d\tau
$$
is the secondary Bott-Chern form. The Donaldson functional (depending on the parameter $s$) is defined as
$$
M(t)= \int_{X\times \{s\}}  R_2(t)\we \omega^n_X,
$$
and for $s\in S'$ we have $M= \displaystyle{\lim_{t \to \infty}} M(t)$.

\begin{gather*}
\omega^{WP}_{S'}= - \frac{1}{2}\int_{X\times S'/S'} {\rm tr}(F_0\we F_0) \we \omega^{n-1} + \hspace{2cm} \\ \hspace{4cm}\lambda \int_{X\times S'/S'} {\rm tr}(F_0)\we \omega^n + \ii \pt \ol\pt M.
\end{gather*}

The integrand  ${\rm tr}(F_0\we F_0) \we \omega^{n-1}$ (as well as ${\rm tr}(F_0)\we \omega^n$) can be estimated in the following way: We apply Hironaka's flattening theorem to the holomorphic map $\wt{X\times S}\times_S\wt S\to \wt S$. After a further blowup ${\wh S} \to \wt S$ the restriction of the pull-back of this map to the proper transform $W\to \wh S$ is a flat (in particular open) holomorphic map. We pull back the integrand to $W$ and denote it by the same letter. Then, on $W$
$$
- \omega_{W, aux}^{n+1} \leq  \frac{1}{2} {\rm tr}(F_0\we F_0) \we \omega^{n-1} \leq  \omega_{W, aux}^{n+1}
$$
for a suitable auxiliary \ka\ form $\omega_{W, aux}^{n+1}$.

The pushforward of  $\omega^{n+1}_{W,aux}$ to $\wh S$ again defines a positive, closed current, which possesses locally a {\em continuous plurisubharmonic $\pt\ol\pt$-potential} by Varouchas' result \cite[Lemme 3.4]{va}.

We have an ''initial'' \wp\ form on $\wh S$ with bounded local potentials
$$
\omega^{WP}_{\wh S, init} =  - \frac{1}{2}\int {\rm tr}(F_0\we F_0) \we \omega^{n-1} + \lambda \int {\rm tr}(F_0)\we \omega^n,
$$
and the correction term $\ii\pt\ol\pt M$, where $M\leq 0$ over $S'$. The extension theorem for bounded plurisubharmonic functions implies the existence of an extension of the \wp\ form to $\wh S$ as a positive, closed current, which can be pushed forward as such to $\wt S$, and finally to $S$.

\subsection{Orbifold case}
The statement of Theorem~\ref{th:ext} would follow from the construction of an initial metric like in Section~\ref{se:initial}, provided a flattening theorem is available in the orbifold case.

\subsection{Extension of the determinant line bundle}
According to \cite[Pro\-position 7]{sch-pos} it is sufficient to show the extension theorem for determinant line bundles on normal complex spaces.  Now the general extension theorem \cite[Theorem II']{sch-pos}, \cite[Theorem 1]{sch-ext} implies the final result. Let $(\lambda,h^Q)$ denote the determinant line bundle on the moduli space $\cM$ of stable holomorphic vector bundles on $X$ given by \eqref{lambda} (cf.\ Theorem~\ref{th:quil}).
\begin{theorem}
The determinant line bundle $\lambda$ extends to a compactification $\ol\cM$ of the moduli space together with the Quillen metric that extends as a singular hermitian metric, in particular $\omega^{WP}$ extends as a positive current that possesses local $\pt\ol\pt$-potentials.
\end{theorem}

\end{document}